\documentclass[letter,11pt]{article}

\usepackage[english]{babel}
\usepackage[utf8]{inputenc}
\usepackage[T1]{fontenc}
\usepackage{microtype} 
\usepackage{lipsum}
\usepackage{tikz-cd}
\usepackage{tikz}
\usepackage{tkz-graph}
\usepackage{amsmath}

\usepackage[linktocpage=true]{hyperref}
\usepackage{csquotes}
\usepackage[capitalize]{cleveref}
\hypersetup{citecolor = black,colorlinks,
			linkcolor = black,
			urlcolor = Color2}

\hypersetup{
  colorlinks=false,
  urlbordercolor=red
}

\crefname{conjecture}{Conjecture}{Conjectures}
\Crefname{conjecture}{Conjecture}{Conjectures}

\usepackage{amsfonts}
\usepackage{mathrsfs}
\usepackage{bbm}
\usepackage{latexsym}
\usepackage{math dots}
\usepackage{amssymb}
\usepackage{mathtools}
\usepackage{amsthm}

\usepackage{authblk}
\usepackage[compact]{titlesec}
\titlespacing*{\section}{0pt}{0.1\baselineskip}{0.2\baselineskip}

\usepackage{lmodern}

\usepackage[margin=.9 in, top=1 in, bottom= .8 in]{geometry}

\makeatletter
\g@addto@macro\normalsize{%
  \setlength\abovedisplayskip{7pt}
  \setlength\belowdisplayskip{7pt}
  \setlength\abovedisplayshortskip{7pt}
  \setlength\belowdisplayshortskip{7pt}
}
\makeatother

\usepackage{tocloft}

\usepackage{amsfonts}
\usepackage{mathrsfs}
\usepackage{bbm}
\usepackage{latexsym}
\usepackage{math dots}
\usepackage{amssymb}
\usepackage{mathtools}

\usepackage{todonotes}
\usepackage{enumitem}
\setlist{nolistsep}
\usepackage{amsthm}
\usepackage{thm-restate}
\usepackage{lipsum}

\usepackage{xcolor}
\usepackage{pagecolor}
\definecolor{Color1}{rgb}{0.0, 0.42, 0.47}
\definecolor{Color2}{rgb}{0.78, 0.11, 0.0}
\definecolor{ocre}{rgb}{0.78, 0.11, 0.0}

\usepackage{titlesec}
\titleformat{\chapter}
  {\huge\scshape\bfseries}
 {Chapter \thechapter:}{.7em}{}
\titlespacing*{\chapter}{0pt}{3.5ex plus 0ex minus 0ex}{1.5ex plus 0ex}

\titleformat{\section}
  {\Large\center\bfseries}
  {\thesection.}{.7em}{}
\titlespacing*{\section}{0pt}{3.5ex plus 0ex minus 0ex}{1.5ex plus 0ex}
\titleformat{\subsection}
  {\large\center\bfseries}
  {\thesubsection.}{.7em}{}
\titlespacing*{\subsection}{0pt}{3.5ex plus 0ex minus 0ex}{1.5ex plus 0ex}
\titleformat{\subsubsection}
  {\bfseries\center}
  {\thesubsubsection.}{.7em}{}
\titlespacing*{\subsubsection}{0pt}{3.5ex plus 0ex minus 0ex}{1.5ex plus 0ex}

\addto\captionsenglish{}

\usepackage{titling}
\usepackage{eso-pic}

\newtheoremstyle{plain}{3mm}{3mm}{\slshape}{}{\bfseries}{.}{.5em}{}
\newtheoremstyle{definition}{2mm}{2mm}{}{}{\bfseries}{.}{.5em}{}
\theoremstyle{plain}
	
\newtheorem{theorem}{Theorem}

\newtheorem{lemma}[theorem]{Lemma}

\theoremstyle{definition}
\newtheorem{definition}[theorem]{Definition}

\newtheorem{question}[theorem]{Question}
\newtheorem{conjecture}{Conjecture}

\theoremstyle{plain}
\newcounter{MainTheoremCounter}

\newtheorem{Maintheorem}[MainTheoremCounter]{Theorem}

\crefname{Maintheorem}{Theorem}{main theorems}
\Crefname{Maintheorem}{Theorem}{Main Theorems}

\theoremstyle{definition}

\numberwithin{equation}{section}

\allowdisplaybreaks

\makeatletter
\renewcommand{\cleardoublepage}{
\clearpage\ifodd\c@page\else
\hbox{}
\vspace*{\fill}
\thispagestyle{empty}
\newpage
\fi}

\renewcommand{\epsilon}{\varepsilon}
\renewcommand{\leq}{\leqslant}
\renewcommand{\geq}{\geqslant}

\renewcommand{\subset}{\subseteq}

\newcommand{\solution}[1]{{\begin{proof}[Solution]\let\qed\relax{#1}\end{proof}}}

\newcommand{\produp}{
  \mathchoice
    {\mathop{\ooalign{$\displaystyle\prod$\cr$\hfil\uparrow\hfil$}}}
    {\mathop{\ooalign{$\textstyle\prod$\cr$\hfil\scriptstyle\uparrow\hfil$}}}
    {\mathop{\ooalign{$\scriptstyle\prod$\cr$\hfil\scriptscriptstyle\uparrow\hfil$}}}
    {\mathop{\ooalign{$\scriptscriptstyle\prod$\cr$\hfil\scriptscriptstyle\uparrow\hfil$}}}
}

\newcommand{\proddown}{
  \mathchoice
    {\mathop{\ooalign{$\displaystyle\prod$\cr$\hfil\displaystyle\downarrow\hfil$}}}
    {\mathop{\ooalign{$\textstyle\prod$\cr$\hfil\scriptstyle\downarrow\hfil$}}}
    {\mathop{\ooalign{$\scriptstyle\prod$\cr$\hfil\scriptscriptstyle\downarrow\hfil$}}}
    {\mathop{\ooalign{$\scriptscriptstyle\prod$\cr$\hfil\scriptscriptstyle\downarrow\hfil$}}}
}

\begin{document}
\pagenumbering{arabic}
\setcounter{page}{1}

\titleformat{\title}
{\centering\LARGE\bfseries}{\MakeUppercase{\title}}{}{}

\titleformat{\author}
{\centering\normalsize\bfseries\MakeUppercase}{\author}{}{}

\let\cleardoublepage\clearpage

\title{Van der Waerden type theorem for amenable groups and FC-groups}
\author{Emilio Parini}
\date{}
\maketitle

\begin{abstract}
    We prove that for a discrete, countable, and amenable group $G$, if the direct product $G^2=G \times G$ is finitely colored then $\{ g \in G : \text{exists } (x,y) \in G^2 \text{ such that } \{ (x,y),(xg,y),(xg,yg)\} \text{ is monochromatic} \}$, is left IP$^{\ast}$. This partially solves a conjecture of V. Bergelson and R. McCutcheon. Moreover, we prove that the result holds for $G^m$ if $G$ is an FC-group, i.e., all conjugacy classes of $G$ are finite.
\end{abstract}

\tableofcontents

\section{Introduction}
In 1928 B. L. van der Waerden proved (see \cite{original-work}) that when partitioning the integers $\mathbb{Z}$ into a finite number of disjoint sets, i.e. $\mathbb{Z}= \bigcup_{j=1}^{r} C_j$, at least one of these sets $C_j$ must contain a finite arithmetic progression of any length. Today van der Waerden's Theorem is one of the cornerstones of Ramsey Theory. This inspired the multidimensional van der Waerden Theorem applicable to $\mathbb{Z}^m$, of which various proofs by different authors were given (see \cite{Rado1945NoteOC}, \cite{Witt}). H. Furstenberg and B. Weiss in \cite{Furstenberg1978TopologicalDA} derived the topological IP-van de Waerden Theorem by introducing a new approach that yields numerous corollaries. Since in this paper, we work with non-abelian groups, we introduce the notion of an IP set in the more general case where $G$ is not necessarily an abelian group: let $(x_n)_{n \in \mathbb{N}}$ be a sequence in $G$. We denote the left finite product set of $(x_n)_{n \in \mathbb{N}}$ with $\mathbf{FP}_L(x_n) =\{ \proddown_{n \in F} x_n : F \subseteq \mathbb{N}, 0 < \left| F \right| < \infty \}$ where $\proddown_{n \in F} x_n$ means that the indices are taken in decreasing order. For example let $G=\left< a,b \mid a^2 = b^2 = e_G \right>$ and let $A_{s},A_{e} \subseteq G$ be the sets of reduced words starting and ending with the letter $a$ respectively. We define similarly $B_s$ and $B_e$. Take $x_1=a$, $x_2=ab$, $x_3=abab$, and for $n \geq 2$, let $x_n=\underbrace{ab ab \ldots ab}_{(n-1)\text{-times}}$.  Then $\mathbf{FP}_L(x_n)=(A_s \cap B_e) \cup (A_s \cap A_e)$. A set $Y \subseteq G$ is called a left IP set if it contains a left finite product set. A set $X \subseteq G$ is called a left IP$^{\ast}$ set if for all left IP sets $Y$ we have $X \cap Y \neq \emptyset$. The refinements of the van der Waerden Theorem, often referred to IP-van der Waerden Theorem, take the following form:
\begin{theorem}[\texorpdfstring{See \cite[Theorem 4.8 on page 44]{Bergelson1996ErgodicRT}}{}] \label{Furstenberg_combinatorial} 
Let $G$ be an abelian group, and let $m \in \mathbb{N}$ be arbitrary. For any finite partition $G^m = \bigcup_{j=1}^{r} C_j$, there exists $1 \leq j \leq r$ such that the set
 \begin{align*}
         \left\{ g \in G \ \middle\vert \begin{array}{l}
     \text{ exists } (x_1,x_2\ldots,x_m) \in G^m \text{ such that } \\
    \{(x_1,x_2\ldots,x_m),(x_1 g, x_2 \ldots,x_m), (x_1 , x_2g, \ldots,x_m), \ldots,(x_1, x_2 , \ldots,x_mg) \} \subseteq C_j
  \end{array}\right\}
     \end{align*}
is an IP$^{\ast}$ set.
\end{theorem}
From the density result in \cite{Austin2013NonconventionalEA} one can replace IP$^{\ast}$ with syndetic in \cref{Furstenberg_combinatorial} (see \cref{syndetic-thick-pws}). \cref{Furstenberg_combinatorial} raises the natural question of which non-abelian groups $G$ this can be extended to. However, in the non-abelian case it is more natural to look at configurations of the following form: 
 \begin{align*}   
    \left\{(x_1,x_2\ldots,x_m),(x_1 g, x_2 \ldots,x_m),(x_1g , x_2g, \ldots,x_m),\ldots,(x_1g, x_2 g, \ldots,x_mg) \right\}. 
     \end{align*}
N. Hindman and V. Bergelson, in their work \cite{free}, provided a counterexample for an infinitely generated free group. Whether these results are valid in amenable groups still remains an open question. In 2007 V. Bergelson and R. McCutcheon proved in \cite{Central} a slightly different result in $G^3$ with a weaker notion of largeness that will be properly defined later, namely the central set, denoted $\mathcal{C}$, and the central$^{\ast}$ set, denoted $\mathcal{C}^{\ast}$. In particular, they obtained that:

\begin{theorem}[\texorpdfstring{See \cite[Theorem 3.1 on page 1260]{Central}}{}]\label{Bergelson_McCutcheon_density}
Let $(X,\mathcal{A},\mu )$ be a Lebesgue probability space and $T,S$ be commuting and measure preserving anti-actions. For any $A \in \mathcal{A}$ with $\mu(A) >0$ there exists $\lambda=\lambda(A,T_g,S_g) >0$ such that \begin{align*}
     \{ g \in G: \mu(A \cap T_g^{-1}A \cap (T_gS_g)^{-1}A) > \lambda \}
 \end{align*}  is both right and left $\mathcal{C}^{\ast}$.
\end{theorem}

With this result, they were able to prove a combinatorial version of van der Waerden's Theorem with $\mathcal{C}^{\ast}$ conclusions, which is formulated as:
\begin{theorem}[\texorpdfstring{See \cite[Theorem 4.4 on page 1270]{Central}}{}] \label{Bergelson_McCutcheon_combinatorial}
    Let $G$ be a countable, discrete, and amenable group. For any finite partition $ G^3 =\bigcup_{j=1}^{r} C_j $, there exists $1 \leq j \leq r$ such that the set
    \begin{align*}
        \{ g \in G: \text{exists } (a,b,c) \in G^3 \text{ such that } \{(a,b,c),(ag,b,c),(ag,bg,c),(ag,bg,cg) \} \subseteq C_j \}
    \end{align*}
    is both right and left $\mathcal{C}^{\ast}$.
\end{theorem}

In \cite{Central}, V. Bergelson and R. McCutcheon conjectured that \cref{Bergelson_McCutcheon_density} also holds for $k$ commuting and measure preserving anti-actions, for $k \in \mathbb{N}$ arbitrary, and this would imply that also \cref{Bergelson_McCutcheon_combinatorial}  holds in $G^m$ for $m \in \mathbb{N}$ arbitrary. Moreover, they conjectured that $\mathcal{C}^{\ast}$ conclusion can be upgraded to IP$^{\ast}$ conclusion. More precisely they have conjectured the following:
\begin{conjecture}\label{conjecture_1}
     Let $(X,\mathcal{A},\mu )$ be a Lebesgue probability space and $T^{(1)},\ldots, T^{(k)}$ be commuting and measure preserving anti-actions. For any $A \in \mathcal{A}$ with $\mu(A) >0$ there exists $\lambda=\lambda(A,T_g^{(1)},\ldots, T_g^{(k)}) >0$ such that \begin{align*}
     \{ g \in G: \mu(A \cap (T_g^{(1)})^{-1}A \cap \ldots \cap (T_g^{(1)} \ldots T_g^{(k)})^{-1}A) > \lambda \}
 \end{align*}  is left IP$^{\ast}$.
\end{conjecture}

\begin{conjecture}\label{conjecture_2}
     Let $G$ be a countable, discrete, and amenable group. Let $m \in \mathbb{N}$ be arbitrary, for any finite partition $ G^m =\bigcup_{j=1}^{r} C_j $, we have that there exists $1 \leq j \leq r$ such that the set
     \begin{align*}
         \left\{ g \in G \ \middle\vert \begin{array}{l}
    \text{exists }  (x_1,x_2\ldots,x_m) \in G^m \text{ such that } \\
    \{(x_1,x_2\ldots,x_m),(x_1 g, x_2 \ldots,x_m),\ldots,(x_1g, x_2 g\ldots,x_mg) \} \subseteq C_j
  \end{array}\right\}
     \end{align*}

    is left IP$^{\ast}$.
\end{conjecture}
 
The purpose of this paper is to prove that \cref{conjecture_2} holds for $m=2$.
To do this we begin by proving that \cref{conjecture_1} holds for $k=1$ and we combine it with a technical color-focusing lemma, namely \cref{thm_7.1.3}. We obtain in \cref{subsec:vdw amenable dimension 2} a combinatorial version of the IP-van der Waerden Theorem for $m=2$. Our new color-focusing lemma allows us to iteratively focus on sets with positive upper density and that if shifted appropriately are nested. At the same time, the color-focusing lemma allows us to keep track of the IP set structure, indeed the shift can be chosen in a specific IP set. Consequently, we establish that: 

\begin{Maintheorem}\label{Main_VdW-2}
Let $G$ be a discrete, countable and amenable group, for any finite partition $ G^2 =\bigcup_{j=1}^{r} C_j $, we have that there exists $1 \leq j \leq r$ such that the set
\begin{align*}
   \{ g \in G : \text{ there exists } (x,y) \in G^2 \text{ such that } \{ (x,y),(x g, y),(x g, y g)\} \subseteq C_j \}
\end{align*}
is left IP$^{\ast}$.
\end{Maintheorem}
An analog result also holds if we consider the following corner set: $\{ (x,y),(gx , y),(gx , gy )\}$. In this case we get that $\{ g \in G : \text{ there exists } (x,y) \in G^2 \text{ such that } \{ (x,y),(gx, y),(gx , gy )\} \subseteq C_i \} $ is a right IP$^{\ast}$ set. Similarly as before, a right IP$^{\ast}$ set contains a right finite product set for some sequence $(x_n)_{n \in \mathbb{N}}$ in $G$, that is a set of the form: $\mathbf{FP}_R(x_n) =\{ \produp_{n \in F} x_n : F \subseteq \mathbb{N}, 0 < \left| F \right| < \infty \}$. Note that in a right finite product set the indices are taken in increasing order. We also prove that \cref{conjecture_2} holds for FC-groups which are groups where each conjugacy class is finite. More precisely, we show that:
  \begin{Maintheorem}\label{van-der-waerden-holds VDW group}
    Let $G$ be an FC-group and $m \in \mathbb{N}$ be arbitrary. For any finite partition $ G^m =\bigcup_{j=1}^{r} C_j $, we have that there exists $1 \leq j \leq r$ such that the set
    \begin{align*}
         \left\{ g \in G \ \middle\vert \begin{array}{l}
    \text{exists } (x_1,x_2\ldots,x_m) \in G^m \text{ such that } \\
    \{(x_1,x_2\ldots,x_m),(x_1 g, x_2 \ldots,x_m),\ldots,(x_1g, x_2 g\ldots,x_mg) \} \subseteq C_j
  \end{array}\right\}
     \end{align*}
    is left IP$^{\ast}$.      
\end{Maintheorem}
The class of FC-groups contains the class of abelian groups, but forms a slightly larger family, including, for example, $\mathbb{Z} \times Q_8$, where $Q_8$ is the quaternion group or the restricted direct product $\bigoplus_{n \in \mathbb{N}} \mathbf{Dih}_n$, where $\mathbf{Dih}_n$ is the dihedral group of order $2n$. In fact, the class of FC-groups is exactly the class of any central extension of torsion-free abelian groups by locally normal groups (see \cite{FC-group}).
In \cref{subsec:vdw group} we introduce the concept of van der Waerden group and we prove that \cref{van-der-waerden-holds VDW group} is true if $G$ is a right van der Waerden group. The result for FC-groups is a special case, indeed thanks to \cref{FC-are VDW groups} we show that FC-groups are van der Waerden groups.


\section{IP-van der Waerden Theorems}
\subsection{Stone-Čech compactification}
Given a nonempty set $G$, a filter on $G$ is a family $p$ of subsets of $G$ satisfying the following:  $\emptyset \not \in p$ and $G \in p$, $p$ is upward closed and $p$ is closed under finite intersection. Moreover, an ultrafilter is a filter which is maximal with respect to inclusion. For a discrete topological space $G$, let $\beta G= \{ p : p \text{ is an ultrafilter on } G \}$. Endowing $\beta G$ with the topology which has $\{ \overline{A} : A \subseteq G\}$ as basis, where $\overline{A}=\{p \in \beta G: A \in p \}$, we have that $\beta G$ is compact Hausdorff and moreover $\beta G$ is a Stone-Čech compactification (see \cite{idem} for more details). We can extend the group operations to $\bullet : \beta G \times \beta G \to \beta G$ making $\beta G$ a left topological compact semigroup. More precisely, for all $p,q \in \beta G$ we have $p \bullet q \in \beta G$ and $A \in p \bullet q $ if and only if $\{x \in G : Ax^{-1} \in p \} \in q$. It is also possible to extend the group operations to $\circ : \beta G \times \beta G \to \beta G$ making $\beta G$ a right topological compact semigroup. In this case $A \in p \circ q$ if and only if $\{ x \in G : x^{-1} A \in q\} \in p$. It is a standard result that a set $Y$ is a left IP set if and only if there exists an idempotent in $(\beta G, \bullet)$ containing $Y$ (see \cite[Theorem 16.4 on page 406]{idem}). Similarly, a set $X$ is a left IP$^{\ast}$ if it is contained in all idempotents (see \cite[Theorem 16.6 on page 406]{idem}). By \cite[Theorem 2.8, on page 42]{idem} we have that $(\beta G,\bullet)$ has a minimal two-sided ideal, denoted by $K(\beta G,\bullet)$. A set is said to be central, denoted by $\mathcal{C}$, if it is contained in some minimal idempotent, i.e. idempotent which is a member of $K(\beta G, \bullet)$. Similarly, a set is said to be central$^{\ast}$, denoted by $\mathcal{C}^{\ast}$, if it is contained in all minimal idempotents. We see that the concept of a central set is weaker than an IP set, since every central set is also an IP set, but not every IP set is a central set. Therefore an IP$^{\ast}$ set is more restrictive, indeed every IP$^{\ast}$ set is also a $\mathcal{C}^{\ast}$ set, but the converse is not true.

\subsection{IP-Poincar\'e Recurrence Theorem}
\label{sec:density version of van der waerden}
If $G$ is a countable group, a sequence $\Phi = (\Phi_n)_{n \in \mathbb{N}}$ of finite subsets of $G$ is said to be a right Følner sequence if for all $g \in G$ we have $\lim_{n \to \infty} \frac{\left| \Phi_n \cap \Phi_n g \right|}{\left| \Phi_n \right|} =1$. A group is right amenable if and only if there exists a right Følner sequence. From this, we can introduce the notions of right upper density of a set. Let $E \subseteq G$ then the right upper density of $E$ is defined to be  $\overline{d}_{\Phi}(E)=\limsup_{n \to \infty} \frac{\left|E \cap \Phi_n \right|}{\left| \Phi_n \right|}$. We notice that the right upper density is right invariant, meaning that $\overline{d}_{\Phi}(E)=\overline{d}_{\Phi}(Eg)$ for all $g \in G$. Moreover, a set is said to be a right positive upper density set if there exists a right Følner sequence $\Phi$ such that $\overline{d}_{\Phi}(E) >0$. In this Section, we prove \cref{thm_6.2} which is \cref{conjecture_1} for $k=1$. We also prove \cref{thm_6.4} which is the density version of \cref{thm_6.2}. 
We assume in all of \cref{sec:density version of van der waerden} that $G$ is a countable, discrete, and amenable group. Moreover, for an ultrafilter $p \in \beta G$ and an indexed sequence $(x_g)_{g \in G}$ in a topological space $X$, we say that $p-\lim_{g \in G} x_g=x$ if and only if for every neighborhood $x \in U$ we have that $\{ g \in G: x_g \in U \} \in p$. Let $(U_g)_{g \in G}$ be a unitary representation of $G$ on a separable Hilbert space $\mathcal{H}$, i.e. a function $g \to U_g$ sending $G$ to unitary operators on $\mathcal{H}$ such that $U_{gh}=U_g U_h$. By Banach-Alaoglu Theorem we have that the unit ball in $\mathcal{H}$ is compact, moreover is also metrizable since $\mathcal{H}$ is separable. Since the unit ball is compact and metrizable, then for any ultrafilter $p \in \beta G$, $p-\lim_g U_g f$ exists weakly for any $f \in \mathcal{H}$.
Henceforth, by $(X,\mathcal{A},\mu)$ we denote a Lebesgue probability space which is measurably isomorphic to a compact metric space equipped with a complete regular Borel measure. Since $(X,\mathcal{A},\mu)$ is a Lebesgue space, then $L^2(X,\mathcal{A},\mu)$ is a separable Hilbert space, hence its unit ball is compact and metrizable in the weak topology.  We also consider $T$ to be a $G$ anti-action which is measure-preserving, i.e. $\mu(T_g^{-1} A)=\mu(A)$ for all $g \in G$ and for all $A \in \mathcal{A}$. It is natural to consider $T$ to be an anti-action, in this way we work with a unitary representation which is an action in \cref{thm_6.1} and another action on $G^2$ in \cref{thm_6.4}.

\begin{theorem}[\texorpdfstring{See \cite[Theorem 2.4 on page 1258]{Central}}{}] \label{thm_6.1}
     Let $(U_g)_{g \in G}$ be a unitary representation of $G$ on a separable Hilbert space $\mathcal{H}$, suppose that $p \in (\beta G,\bullet)$ is an idempotent and for $f \in \mathcal{H}$, let $P f = p-\lim_g U_g f$ weakly. Then $P$ is the orthogonal projection onto a closed subspace of $\mathcal{H}$.
 \end{theorem} 

 Note that in the original version of \cref{thm_6.1}, V. Bergelson and R. McCutcheon assume $p$ to be a minimal idempotent to conclude that for all $g \in G$ and all $f \in \mathcal{H}$ we have $P U_g f = U_g P f$. They require this because they use two commuting anti-actions. Given that we only have one anti-action $T$, we can drop this assumption since we do not need the orthogonal projection $P$ defined in \cref{thm_6.1} to commute with the unitary representation of the Hilbert space $\mathcal{H}$. It is only important that $P$ is an orthogonal projection. We can conclude this even if the ultrafilter is an arbitrary idempotent. From this, we can upgrade the conclusion from left $\mathcal{C}^{\ast}$ to left IP$^{\ast}$ and obtain \cref{thm_6.2}. In the proof of \cref{thm_6.1} at some point there is the need to prove that $P=P^2$. I would like to thank E. Ackelsberg to help me notice the following interesting fact. Let $p \in (\beta G,\bullet)$ be an idempotent, let $p^{-1} =\{ A^{-1} : A \in p\}$. Note that $p^{-1} \in (\beta G, \circ)$. If we consider $P^{\ast}$ to be the adjoint operator of $P$, that is $P^{\ast}f=p^{-1}- \lim_g U_g^{\ast} f$. Then is easier to get $P^{\ast}=P^{\ast} P^{\ast}$. Indeed 
\begin{align*}
    P^{\ast} P^{\ast} f &= p^{-1}-\lim_{h} \left( p^{-1}-\lim_g U_g^{\ast} U_h^{\ast} f \right)\\
    &=p^{-1}-\lim_{h} \left( p^{-1}-\lim_g U_{gh}^{\ast} f \right)\\
    &=P^{\ast} f.
\end{align*}
This argument fails with $P$ since $U$ is an action. Surprisingly, $P$ is self-adjoint, i.e. $P^{\ast}=P$. This follows from the fact that the orthogonal projection $P$ is idempotent and has operator norm bounded by $1$. Thanks to the this fact one can use a modified version of \cref{thm_6.1} to obtain in \cref{thm_6.2} also right IP$^{\ast}$.

 \begin{theorem}\label{thm_6.2}
    Let $G$ be a countable, discrete and amenable group. Let $T$ be a $G$-anti-action on $X$. For any $A \in \mathcal{A}$ with $\mu(A) >0$ we have that 
    \begin{align*}
        \{ g \in G : \mu(A \cap T_g^{-1} A) >0 \}
    \end{align*}
    is both right and left IP$^{\ast}$. 
 \end{theorem}
 \begin{proof}
     Take any arbitrary idempotent $p \in (\beta G,\bullet)$. Then the anti-action $T$ gives rise to a unitary $G$-action on $L^2(X,\mathcal{A},\mu)$ by $T_g f(x) = f(T_g x)$. Let $Pf=p-\lim_g T_g f$ be the corresponding orthogonal projection, then we have that
     \begin{align*}
         p-\lim_g \mu(A \cap T_{g}^{-1} A) &= \int \mathbf{1}_A P \mathbf{1}_A d \mu \\
         &= \int ( P \mathbf{1}_A)^2 d \mu \\
         & \geq \left( \int P \mathbf{1}_A d \mu \right)^2 = \mu^2(A) >0 
      \end{align*} 
    that is, we have $\{ g \in G : \mu(A \cap T_g^{-1} A) >0 \} \in p$ and since $p $ was an arbitrary idempotent ultrafilter we deduce that $ \{ g \in G : \mu(A \cap T_g^{-1} A) >0 \}$ is a left IP$^{\ast}$ set. Recall that $ T$ is an anti-action, hence $T_{gh} x = T_{h} T_{g} x$. Moreover $T_{gh}^{-1} A = T_{g}^{-1} T_{h}^{-1} A$. Let $(g_n)_{n \in \mathbb{N}}$ be any sequence in $G$. Since $T$ is measure-preserving for all $k \in \mathbb{N} $ we have that $\mu(A)= \mu(T_{g_1 \ldots g_{k-1} g_k}^{-1} A) $. Choosing $k$ sufficiently large, by pigeonhole principle we can find $1 \leq i<j\leq k$ such that  
    \begin{align*}
        \mu(T_{g_1 g_2 \ldots  g_i}^{-1} A \cap T_{g_1 g_2 \ldots g_j}^{-1} A) >0.
    \end{align*} 
    Thus we get
    \begin{align*}
        \mu(T_{g_1 g_2 \ldots  g_i}^{-1} A \cap T_{g_1 g_2 \ldots g_j}^{-1} A) 
        &=\mu( A \cap T_{g_{i+1} \ldots g_j}^{-1}  A) >0 
    \end{align*}
    and $g_{i+1} \cdot \ldots \cdot g_j \in \mathbf{FP}_R(g_n)$. Hence
     \begin{align*}
        \{ g \in G : \mu(A \cap T_g^{-1} A) >0 \}
    \end{align*}
    is a right IP$^{\ast}$ set.
 \end{proof}

 I would like to thank V. Bergelson for helping me notice that the conclusion of \cref{thm_6.2} can be upgraded to right $\Delta^{\ast}$, that is the set $ \{ g \in G : \mu(A \cap T_g^{-1} A) >0 \}$ intersects every set of differences $B^{-1}B$ for $B \subseteq G$. Take $g_1,\ldots,g_k \in B$ with $k$ sufficiently large. There exists $1 \leq i < j \leq k$ such that $\mu(T_{g_i}^{-1} A \cap T_{g_j}^{-1} A) >0$. Hence $\mu( A \cap T_{g_i^{-1} g_j}^{-1} A ) >0$. \newline
\cref{thm_6.4} is essentially the same as \cite[Theorem 4.2, on page 1268]{Central}, the only thing that changes is that we use \cref{thm_6.2} to have that $\{ g: \overline{d}_{\Phi}(E \cap E(g^{-1},e_G)) >0 \}$ is an IP$^{\ast}$ set. V. Bergelson and R. McCutcheon use \cref{Bergelson_McCutcheon_density} to conclude that $\{ g: \overline{d}_{\Phi}(E \cap E(g^{-1},e_G) \cap E(g^{-1},g^{-1}) >0 \}$ is a $\mathcal{C}^{\ast}$ set.
\begin{theorem}\label{thm_6.4}
     Let $G$ be a countable amenable group with identity $e_G$ and suppose $E \subseteq G \times G$ has a positive upper density with respect to some right Følner sequence $\Phi= (\Phi_n)_{n \in \mathbb{N}}$. Then with respect to $(\Phi_n)_{n \in \mathbb{N}}$, we have that 
    \begin{align*}
        \{ g \in G: \overline{d}_{\Phi}(E \cap E(g^{-1},e_G)) > 0 \}
    \end{align*}
    is both right and left IP$^{\ast}$.
\end{theorem}

\begin{proof}
    To prove the result we apply Furstenberg correspondence principle (see \cite[Proposition 4.1 on page 1267]{Central}). Let $\mathcal{L}_{(g,h)}$ be the action on $G^2$ defined by left multiplication by $(g,h) \in G^2$, that is \begin{align*}
        \mathcal{L}_{(g,h)} (a,b) = (ga,hb).
    \end{align*} 
    Let $ \Omega = \{0,1\}^{G \times G}$, define the following anti-action on $\Omega$ by \begin{align*}
        U_{(g,h)} \xi(a,b) &= \xi \circ \mathcal{L}_{(g,h)}(a,b)=\xi(ga,hb),
    \end{align*} for $\xi \in \Omega$ and $(a,b) \in G^2$. Indeed we have that 
    \begin{align*}
        U_{(gg',hh')} \xi &= \xi \circ  \mathcal{L}_{(gg',hh')} \\
        &=U_{(g',h')} U_{(g,h)} \xi
    \end{align*}
    Let's take $ \xi = \mathbf{1}_E \in \Omega$, where $E \subseteq G^2$ is a set of positive upper density with respect to some right Følner sequence $\Phi= (\Phi_n)_{n \in \mathbb{N}}$.  Let's define $X = \overline{ \{ U_{(g,h)} \xi : (g,h) \in G^2 \} }$. Moreover, we define the following set $A = \{ \eta \in X : \eta(e_G,e_G)=1\} \subseteq X$. According to Furstenberg correspondence principle we have that there exists a $\{U_{(g,h)}\}-$invariant measure $\mu$ on $X$ with $\mu(A) = \overline{d}_{\Phi}(E) >0$ such that for every $(g_1,h_1),\ldots, (g_k,h_k) \in G^2$ we have that
    \begin{align*}
        \mu( U_{(g_1,h_1)}^{-1} A \cap \ldots \cap U_{(g_k,h_k)}^{-1} A ) \leq \overline{d}_{\Phi}(E(g_1,h_1)^{-1} \cap \ldots \cap E(g_k,h_k)^{-1}).
    \end{align*}
    Now let $T_g:= U_{(g,e_G)}$ for every $g \in G$. We have that $T$ is a measure-preserving anti-action.
    It follows from \cref{thm_6.2} that
    \begin{align*}
       \{ g \in G : \mu(A \cap T_g^{-1} A)>0 \}
    \end{align*}
    
    is both right and left IP$^{\ast}$. 
    We notice that
    \begin{align*}
        \mu(A \cap T_g^{-1} A) = \mu(A \cap U_{(g,e_G)}^{-1} A ) \leq \overline{d}_{\Phi}(E \cap E(g^{-1},e_G) ).
    \end{align*}
    Hence we have that 
     \begin{align*}
       \{ g \in G : \mu(A \cap T_g^{-1} A)>0 \} \subseteq \{ g \in G: \overline{d}_{\Phi}(E \cap E(g^{-1},e_G)) > 0 \}
    \end{align*}
    and thus it is both a right and left IP$^{\ast}$ set.
\end{proof}  

\subsection{IP-van der Waerden Theorems for amenable groups and FC-groups}
\label{sec:van der waerden}

In \cref{subsec:vdw amenable dimension 2} we prove \cref{Main_VdW-2}. In \cref{subsec:vdw group} we prove \cref{van-der-waerden-holds VDW group}. Before we go into the details of these proofs, let us introduce some notations that will simplify and lighten the discussion and reading moving forward. Let's denote by $G^m$ the direct product of $G$ with itself $m$-times. Let $\mathcal{R}_g^{(j)}$ be the multiplication on the right by $g$ on the $j$-th component. Similarly by $\mathcal{L}_g^{(j)}$ let's denote the multiplication on the left by $g$ on the $j$-th component. Moreover, let $ G^m=\bigcup_{j=1}^{r}C_j$ be any finite pairwise disjoint partition of $G^m$. We say that the function $\chi : G^m \to \{1 ,\ldots , r\} $  defined by $x \mapsto \chi(x)=j$ if and only if $x \in C_j$ is a finite coloring of $G^m$.

\begin{definition}
Let $G$ be a group. Let $1 \leq j_1< j_2 < \ldots < j_k \leq m$ and let $g \in G$ and $x \in G^m$. We define the right corner sets by
\begin{align*}
    \mathfrak{R}_{x, g, (j_1,\ldots,j_k)}=\{ x , \mathcal{R}_g^{(j_1)}x, \mathcal{R}_g^{(j_1)}\mathcal{R}_g^{(j_2)}x,\ldots, \mathcal{R}_g^{(j_1)} \ldots \mathcal{R}_g^{(j_k)} x\}
\end{align*} 
\end{definition}
We define similarly the left corner sets, denoted by $\mathfrak{L}_{x, g, (j_1,\ldots,j_k)}$, by replacing each $\mathcal{R}^{(j)}$ with $\mathcal{L}^{(j)}$.

\begin{definition}
   Let $1\leq k \leq m$. Let $\chi : G^m \to \{1 ,\ldots , r\}$ be any finite coloring. Then
   \begin{enumerate}
       \item $D_k(\chi) = \begin{Bmatrix} g \in G: \text{exsits } 1 \leq j \leq r, \text{ and } x \in G^m \text{ such that } \mathfrak{R}_{x, g, (1,\ldots,k)} \subseteq C_j \end{Bmatrix}.$
       \item $\tilde{D}_k(\chi) = \begin{Bmatrix} g \in G: \text{exsits } 1 \leq j \leq r, \text{ and } x \in G^m \text{ such that } \mathfrak{R}_{x, g, (2,\ldots,k+1)} \subseteq C_j \end{Bmatrix}.$
   \end{enumerate}
\end{definition}

\cref{lemma_20.1} is useful to prove \cref{Main_VdW-2}, and it says that if $D_k(\chi)$ is IP$^{\ast}$ for every finite coloring $\chi $ then we can permute the coordinates and find that $\tilde{D}_k(\chi)$ is also IP$^{\ast}$ for every finite coloring.
\begin{lemma}\label{lemma_20.1}
    Let $1 \leq k < m$. Let $G$ be a discrete and countable group. We have that $D_k(\chi)$ is a left IP$^{\ast}$ set for all $\chi$ if and only if $\tilde{D}_k(\chi)$ is a left IP$^{\ast}$ set for all $ \chi $.
\end{lemma}
\begin{proof}
    Assume that $D_k(\chi)$ is left IP$^{\ast}$ for all $\chi$. Consider any arbitrary finite coloring $\chi : G^m \to \{1, \ldots, r \}$ defined such that $\chi(x) = j$ if and only if $x \in C_j$. Let $\tau_k$ be the map that shifts the first $k+1$ components of $G^m$, more precisely let $\sigma_k=(1,k+1,k,k-1,\ldots,3,2)$, and define $\tau_{k} : G^m \to G^m$ as
    \begin{align*}
      g=  (g_j)_{1 \leq j \leq m} \mapsto \tau_{k}(g)= (g_{\sigma_{k}(j)})_{1 \leq j \leq m}
    \end{align*}
    note that by defining $\sigma_k^{-1} = (1, 2, 3, \ldots, k, k+1)$, we have the inverse map $\tau_{k}^{-1} : G^m \to G^m$ defined as
    \begin{align*}
        g=  (g_j)_{1 \leq j \leq m} \mapsto \tau_{k}^{-1}(g)= (g_{\sigma_{k}^{-1}(j)})_{1 \leq j \leq m}.
    \end{align*}
    Define the finite coloring $\tilde{\chi}: G^m \to \{1, \ldots, r \}$, where $\tilde{\chi}(x)=j$ if and only if $x \in \tau_{k}(C_j)$. Since $\tilde{\chi}$ is a finite coloring we have that $D_k(\tilde{\chi})$ is left IP$^{\ast}$. We deduce then 
    \begin{align*}
        D_k(\tilde{\chi}) & = \begin{Bmatrix}    
        g \in G:  \text{exsits } 1 \leq j \leq r, \text{ and } \tau_{k}(x) \in G^m \text{ such that } \mathfrak{R}_{\tau_{k}(x), g, (1,\ldots,k)} \subseteq \tau_{k}(C_j) \end{Bmatrix}  \\
    & = \begin{Bmatrix} 
        g \in G:  \text{exsits } 1 \leq j \leq r, \text{ and } x \in G^m \text{ such that } \mathfrak{R}_{x, g, (2,\ldots,k+1)} \subseteq C_j \end{Bmatrix}  \\
        &= \tilde{D}_k(\chi).
    \end{align*}
   The other direction is done similarly.
\end{proof}

With the introduced notation, we can rewrite \cref{conjecture_2} as:
\begin{conjecture}
    Let $G$ be a countable, discrete, and amenable group and $m \in \mathbb{N}$ arbitrary. Let $\chi : G^m \to \{1 ,\ldots , r\}$ be any given finite coloring. Then $D_m(\chi)$ is a left IP$^{\ast}$ set.
\end{conjecture}

\subsection{IP-van der Waerden Theorem for amenable groups}
\label{subsec:vdw amenable dimension 2}

 An immediate consequence of \cref{thm_6.4} is that for any finite coloring $\chi$ of $G^2$, the set
 \begin{align*}
     D_1=\{ g \in G : \text{ such that exists } (a,b) \in G^2 \text{ such that } \{ (a,b), (ag,b) \} \text{ is } \chi-\text{monochormatic}\}
 \end{align*}
 is left IP$^{\ast}$. Indeed given any finite coloring of $G^2$, since the family of positive upper-density set is partition regular, it follows that one of the colors has a positive-upper density. We notice that \cref{thm_6.4} is stronger than just $D_1$ or $D_1(\chi)$ to be a left IP$^{\ast}$ set. For each set of positive upper density $A$ we can find a set of positive upper density $B(g)=A \cap A(g^{-1},e_G)$, with $g$ that can be chosen in an IP$^{\ast}$ set $D$. We can think of $D$ as the set $\{g: \overline{d}_{\Phi}(A \cap A(g^{-1},e_G)) >0 \}$ given in \cref{thm_6.4}. In other words for each IP set $X$ we can find an IP set $Y = D \cap X \subseteq X$ such that for each $g \in Y$ we have a set of positive upper density $B(g)$ such that for each $(a,b) \in B(g)$ we have that $\{ (a,b),(ag,b) \} \subseteq A$. By swapping the coordinates and translating $B(g)$ by $(g^{-1},e_G)$, we may assume that there exists a set of positive upper density $\tilde{B}(g)$ such that for each $(a,b) \in \tilde{B}(g)$ we have that $\{ (ag,b),(ag,bg) \} \subseteq A$.
 Given this we can derive \cref{thm_7.1.3}, which is a color-focusing lemma. More precisely, for any $r$ we can iteratively focus ourselves on sets $A_0,A_1,\ldots, A_r$ such that each has a positive upper density. Moreover, every $A_j$ is $\chi$-monochromatic, possibly with different colors. Furthermore for each IP set $X$, we find a nested sequence of IP sets $Y_{r-1} \subseteq \ldots \subseteq Y_1 \subseteq Y_0 \subseteq X$ of a specific form, such that for each $0 \leq i < j \leq r$ there will be a shift $u_{j,i} \in Y_i$ with the following property: for each $(x,y) \in A_j$ we have that $\{ (x u_{j,i}, y),(x u_{j,i}, y u_{j,i} ) \} \subseteq A_i$. This color-focusing lemma allows us to focus on an arbitrary number of monochromatic sets $A_0,\ldots, A_r$ with the property just described. In \cref{thm_7.1.4}, for arbitrariness of the IP set $X$ initially considered, we deduce that $D_2(\chi)$ is left IP$^{\ast}$. Indeed we can find a $\chi$-monochromatic configuration of the form $\{ (x,y),(x u_{j,i}, y),(x u_{j,i}, y u_{j,i} ) \} $, with $u_{j,i}$ lying in an arbitrary IP set $X$. This argument is valid also in $G^m$ with $m \in \mathbb{N}$. What is no longer guaranteed when $m 
 >2$ is that for any $A$ of positive upper density, we can find a set of positive upper density $B(g)=A \cap A(g^{-1},e_G, \ldots , e_G) \cap \ldots \cap A(g^{-1}, g^{-1},\ldots, g^{-1}, e_G)$. Indeed we have derived \cref{thm_6.4} only for $G^2$.

\begin{lemma}\label{thm_7.1.2}
Let $G$ be an amenable group. Let $\chi : G^m \to \{1,\ldots r\}$ be any finite coloring and let $1\leq k-1 < m$. Let $X$ be any left IP set. If for all right positive upper density sets $A \subseteq G^m$ there exists a left IP set $Y=X \cap D \subseteq X$, where $D$ is some left IP$^{\ast}$, such that for all $g \in Y$ there exists a set of right positive upper density $B=B(g)$ such that for all $b \in B$ we have that
    \begin{align*}
        \mathfrak{R}_{B,g,(1,\ldots,k-1)} \subseteq A.
    \end{align*}
    Then we have that there exists a set of right positive upper density $\tilde{B}=\tilde{B}(g)$ such that for all $\tilde{b} \in B$
    \begin{align*}
        \{ \mathcal{R}_g^{(1)} \tilde{b}, \mathcal{R}_g^{(1)} \mathcal{R}_g^{(2)} \tilde{b}, \ldots, \mathcal{R}_g^{(1)}\ldots \mathcal{R}_g^{(k)} \tilde{b} \}  \subseteq A
    \end{align*}
\end{lemma}
\begin{proof}
  Let $\tau_{k-1}$ be the bijection by shifting the first $k$ coordinates in $G^m$ defined in the proof of \cref{lemma_20.1}. Since $A$ is a right positive upper density set if and only if $\tau_{k-1}(A)$ is a right positive upper density set, we can also find a set of right positive upper density $\tau_{k-1}(B)$ such that for all $b \in \tau_{k-1}(B)$ we have that 
    \begin{align*}
    \{ b, \mathcal{R}_g^{(1)}b, \ldots, \mathcal{R}_g^{(1)}\ldots \mathcal{R}_g^{(k-1)} b \} \subseteq \tau_{k-1}(A).
\end{align*}
Then there exists a set of right positive upper density $B$ such that for all $b \in B$ we have that 
\begin{align*}
    \{ b, \mathcal{R}_g^{(2)}b, \ldots, \mathcal{R}_g^{(1)}\ldots \mathcal{R}_g^{(k)} b \} \subseteq A.
\end{align*}
Let $ \tilde{B} = \mathcal{R}_{g^{-1}}^{(1)} B$, we have that $\tilde{B}$ is a set of right positive upper density, since the density with respect to some right Følner sequence is right invariant. Using the fact that $\mathcal{R}^{(1)}$ commute with $\mathcal{R}^{(j)}$ for $j > 1$, then for all $\tilde{b} \in \tilde{B}$ we get that
\begin{align*}
    \{ \mathcal{R}_g^{(1)} \tilde{b}, \mathcal{R}_g^{(1)} \mathcal{R}_g^{(2)} \tilde{b}, \ldots, \mathcal{R}_g^{(1)}\ldots \mathcal{R}_g^{(k)} \tilde{b} \} \subseteq A.
\end{align*}
This concludes the proof.
\end{proof}
Let $(x_n)_{n \in \mathbb{N}}$ be a sequence in $G$. A sequence $(y_n)_{n \in \mathbb{N}}$ is a product subsystem of $(x_n)_{n \in \mathbb{N}}$ if and only if there is a sequence of sets $(F(n))_{n\in \mathbb{N}}$, with $F(n) \subseteq \mathbb{N}$ and $0 < \left| F(n) \right| < \infty$ such that for every $n \in \mathbb{N}$ we have that $\max F(n) < \min F(n+1)$ and
\begin{align*}
    y_n = \proddown_{f \in F(n)} x_f.
\end{align*}
Let $r \in \mathbb{N} $ be arbitrary. Let $\mathbf{FP}_L(y_n^{(r-1)}) \subseteq \mathbf{FP}_L(y_n^{(r-2)}) \subseteq  \ldots \subseteq  \mathbf{FP}_L(y_n^{(0)}) $ be such that for all $0 < j \leq r-1$ we have that $(y_n^{(j)})_{n \in \mathbb{N}}$ is a product subsystem of $(y_n^{(j-1)})_{n \in \mathbb{N}}$. Then for any given $x_0,\ldots,x_{r-2}$, such that $x_j \in \mathbf{FP}_L(y_n^{(j)})$ for all $0 \leq j \leq r-2$, it is straightforward to check that there exist a global $x \in \mathbf{FP}_L(y_n^{(r-1)})$ such that for all $0 \leq j \leq r-2$ we have $xx_j \in \mathbf{FP}_L(y_n^{(j)})$. 
\begin{theorem}[Color focusing]\label{thm_7.1.3}
    Let $1 \leq k-1 <m$ and $r \in \mathbb{N}$. Let  $\chi$ be any finite coloring of $G^m$.  Assume that for all right positive upper density sets $A \subseteq G^m$ and for any left IP set $X$ there exists a left IP set $Y=X \cap D \subseteq X$, where $D$ is some left IP$^{\ast}$, such that for all $g \in Y$ there exists a set of right positive upper density $B=B(g)$ such that for all $b \in B$ we have that
    \begin{align*}
        \mathfrak{R}_{B,g,(1,\ldots,k-1)} \subseteq A.
    \end{align*}
    Then there exists $\chi$-monochromatic sets of right positive upper density $ A_0,\ldots, A_r \subseteq G^m$ and left IP sets $ Y_{r-1} \subset \ldots \subset Y_0 \subset X$ such that for all $0 \leq i < j \leq r$ there exists $u_{j,i} \in Y_i$ with
    \begin{align*}
        \{ \mathcal{R}_{u_{j,i}}^{(1)} a , \ldots , \mathcal{R}_{u_{j,i}}^{(1)} \ldots \mathcal{R}_{u_{j,i}}^{(k)} a : a \in A_j \} \subseteq A_i.
    \end{align*}
\end{theorem}
\begin{proof}
    We proceed by induction on $r$. If $r=1$ it follows from the fact that sets of upper positive density are partition regulars, indeed there exists a monochromatic set of upper positive density $A_0$. Thanks to \cref{thm_7.1.2} we can find a left IP set $Y_0 \subset X$ such that for all $g_1 \in Y_0$ we have that there exists a set of positive upper density $B_1=B_1(g_1)$ such that for all $b \in B_1$ we have 
    \begin{align*}
         \{ \mathcal{R}_{g_1}^{(1)} b , \ldots , \mathcal{R}_{g_1}^{(1)} \ldots \mathcal{R}_{g_1}^{(k)} b \} \subseteq A_{0}.
    \end{align*}
    Up to replace $Y_0$ with a left IP subset of the form $\mathbf{FP}_L(y_n^{(0)})$, we may assume that $Y_0 = \mathbf{FP}_L(y_n^{(0)})$. Let's choose and fix $u_{1,0} \in Y_0$ arbitrarily and find the corresponding set of positive upper density $B_1=B_1(u_{1,0})$ with the property described above. The finite coloring on $G^m$ induces a finite coloring of $B_1$. Again using the partition regularity we can find a monochromatic set of positive upper density $A_1 \subseteq B_1$ such that
    \begin{align*}
         \{ \mathcal{R}_{u_{1,0}}^{(1)} a , \ldots , \mathcal{R}_{u_{1,0}}^{(1)} \ldots \mathcal{R}_{u_{1,0}}^{(k)} a: a \in A_1 \} \subseteq A_{0}.
    \end{align*}
    If $r \geq 2$ assume we have $A_0, \ldots, A_{r-1}$, and $Y_{r-2}=\mathbf{FP}_L(y_n^{(r-2)}) \subset \ldots \subset Y_0 = \mathbf{FP}_L(y_n^{(0)})$, where for all $ 1 \leq j \leq  r-2$ we have that $(y_n^{(j)})_{n \in \mathbb{N}}$ is a product subsystem of $(y_n^{(j-1)})_{n \in \mathbb{N}}$. Then $A_r$ and $Y_{r-1}$ and are constructed as follows: using \cref{thm_7.1.2} with $Y_{r-2}$ and $A_{r-1}$ we have that there exists a left IP set $Y := Y_{r-2} \cap D \subset Y_{r-2}$ such that for all $g \in Y $ there exists a set of positive upper density $B=B(g) $ such that for all $b \in B$ we have
    \begin{align*}
         \{ \mathcal{R}_{g}^{(1)} b , \ldots , \mathcal{R}_{g}^{(1)} \ldots \mathcal{R}_{g}^{(k)} b \} \subseteq A_{r-1}.
    \end{align*}
    Notice that
    \begin{align*}
       \bigcap_{m=1}^{\infty} \overline{ \mathbf{FP}_L(y_n^{(r-2)} : n \geq m)} 
    \end{align*}
     is a subsemigroup of $(\beta G, \bullet)$ and there exists an idempotent $p \in (\beta G,\bullet)$ such that for all $m \in \mathbb{N}$ we have that $\mathbf{FP}_L(y_n^{(r-2)} : n \geq m) \in p$ (see \cite[Lemma 5.11 on page 112]{idem}). Moreover, since $D$ is a left IP$^{\ast}$ we have that $D \in p$, thus we have $Y=D \cap Y_{r-2} \in p$ since ultrafilters are closed under finite intersections. So we can find a product subsystem $(y_n^{(r-1)})_{n \in \mathbb{N}}$ of $(y_n^{(r-2)})_{n \in \mathbb{N}}$ such that $Y_{r-1} := \mathbf{FP}_L(y_n^{(r-1)}) \subseteq Y \subseteq Y_{r-2} $ (see \cite[Theorem 5.14 on page 113]{idem}). Hence we have that there exists a left IP set $Y_{r-1} \subset Y_{r-2}$ such that for all $g_r \in Y_{r-1} $ there exists a set of positive upper density $B_r=B_r(g_r) $ such that for all $b \in B_r$ we have that
    \begin{align*}
         \{ \mathcal{R}_{g_r}^{(1)} b , \ldots , \mathcal{R}_{g_r}^{(1)} \ldots \mathcal{R}_{g_r}^{(k)} b \} \subseteq A_{r-1}
    \end{align*}
     Thanks to the induction hypothesis for all $1 \leq i < r-1$ we already have found $u_{r-1,i} \in Y_i $, moreover we can choose $u_{r,r-1} \in Y_{r-1}$ such that for all $1 \leq i < r-1$, we have that $u_{r,i}:=u_{r,r-1} u_{r-1,i} \in Y_i$. Let find the corresponding $B_r=B_r(u_{r,r-1})$.
    The finite coloring on $G^m$ induces a finite partition of $B_r$. Using the partition regularity we can find a monochromatic set of positive upper density $A_r \subseteq B_r$ such that
    \begin{align*}
         \{ \mathcal{R}_{u_{r,r-1}}^{(1)} a , \ldots , \mathcal{R}_{u_{r,r-1}}^{(1)} \ldots \mathcal{R}_{u_{r,r-1}}^{(k)} a: a \in A_r \} \subseteq A_{r-1}.
    \end{align*}
    Let $0 \leq i < j \leq r$. If $j < r$ then the results follow from the induction hypothesis. If $j=r$ then by induction hypothesis we have that
    \begin{align*}
         \{ \mathcal{R}_{u_{r-1,i}}^{(1)} a , \ldots , \mathcal{R}_{u_{r-1,i}}^{(1)} \ldots \mathcal{R}_{u_{r-1,i}}^{(k)} a : a \in A_{r-1} \} \subseteq A_{i}
    \end{align*}
    then since $u_{r,i}=u_{r,r-1} u_{r-1,i} \in Y_i$, it follows that
     \begin{align*}
         \{ \mathcal{R}_{u_{r,i}}^{(1)} a , \ldots , \mathcal{R}_{u_{r,i}}^{(1)} \ldots \mathcal{R}_{u_{r,i}}^{(k)} a: a \in A_r \}
         & \subseteq \{ \mathcal{R}_{u_{r-1,i}}^{(1)} a , \ldots , \mathcal{R}_{u_{r-1,i}}^{(1)} \ldots \mathcal{R}_{u_{r-1,i}}^{(k)} a :  a \in A_{r-1} \} \\
         & \subseteq A_i.
    \end{align*}
\end{proof} 
Thanks to the color focus lemma, we can deduce the following
\begin{theorem}\label{thm_7.1.4}
    Let $1 \leq k-1 <m$. Let  $\chi: G^m \to \{1,\ldots,r\}$ be any finite coloring of $G^m$. Assume that for all right positive upper density sets $A \subseteq G^m$ and that for any left IP set $X$ there exists a left IP set $Y=X \cap D \subseteq X$, where $D$ is some left IP$^{\ast}$, such that for all $g \in Y$ there exists a set of right positive upper density $B=B(g)$ such that for all $b \in B$ we have that
    \begin{align*}
        \mathfrak{R}_{B,g,(1,\ldots,k-1)} \subseteq A.
    \end{align*}
    Then we have that $D_k(\chi)$ is a left IP$^{\ast}$ set.
\end{theorem}
\begin{proof}
        Let $X$ be any left IP set, using \cref{thm_7.1.3} with $r$ equal to the number of colors used, then there exists sets of right positive upper-density $ A_0,\ldots, A_r \subseteq G^m$ and there exists left IP sets $ Y_{r-1} \subset \ldots \subset Y_0 \subset X$ such that for all $0 \leq i < j \leq r$ there exists $u_{j,i} \in Y_i$ with
    \begin{align*}
        \{ \mathcal{R}_{u_{j,i}}^{(1)} a , \ldots , \mathcal{R}_{u_{j,i}}^{(1)} \ldots \mathcal{R}_{u_{j,i}}^{(k)} a : a \in A_j \} \subseteq A_i.
    \end{align*}
    Since there are $r+1$ sets but only $r$ colors we deduce by pigeonhole principle that two sets must have the same color. Hence there exists $0 \leq i < j \leq r$ such that $A_i$ and $A_j$ have the same color. So we have that
    \begin{align*}
        \{a, \mathcal{R}_{u_{j,i}}^{(1)} a , \ldots , \mathcal{R}_{u_{j,i}}^{(1)} \ldots \mathcal{R}_{u_{j,i}}^{(k)} a : a \in A_j \} .
    \end{align*}
    is $\chi$-monochromatic. So $\{u_{j,i}\} \subseteq D_{k}(\chi) \cap X \neq \emptyset$ and since $X$ was an arbitrary left IP set we have that $D_{k}(\chi)$ is a left IP$^{\ast}$ set.
    \end{proof}
    
We assumed that for all right positive upper density sets $A \subseteq G^m$ and that for any left IP set $X$ there exists a left IP set $Y=X \cap D \subseteq X$, where $D$ is some left IP$^{\ast}$, such that for all $g \in Y$ there exists a set of right positive upper density $B=B(g)$ such that for all $b \in B$ we have that
    \begin{align*}
        \mathfrak{R}_{B,g,(1,\ldots,k-1)} \subseteq A.
    \end{align*}
    Below, in \cref{thm_7.1.5} we prove that this is true in $G^2$ and $k-1=1$, as a consequence of \cref{thm_6.4}.
\begin{theorem}\label{thm_7.1.5}
    Let $G$ be a countable, discrete and amenable group. For all right positive upper density set $A \subseteq G^2$ and for any left IP set $X$ there exists a left IP set $Y=X \cap D \subseteq X$, where $D$ is some left IP$^{\ast}$, such that for all $g \in Y$ there exists a set of right positive upper density $B=B(g)$ such that for all $b \in B$ we have that
    \begin{align*}
        \mathfrak{R}_{B,g,(1)} \subseteq A.
    \end{align*}
\end{theorem}

\begin{proof}
    Since $A$ is of right positive upper density there exists a right Følner sequence $\Phi$ such that $\overline{d}_{\Phi}(A) >0$. From \cref{thm_6.4} we have that \begin{align*} D=\{ g \in G : \overline{d}_{\Phi}(A \cap A(g^{-1},e_G)) > 0 \} \end{align*} is left IP$^{\ast}$. Let $X$ be any left IP set, we have that $Y=D\cap X$ is a left IP set, and for all $g \in Y$ we have that $B(g)= A \cap A(g^{-1},e_G)$ is a set of right positive upper density and thus we have, as desired, that for all $b \in B$ 
    \begin{align*}
      \mathfrak{R}_{b,g,(1)}=  \{ b, \mathcal{R}_g^{(1)}b  \} \subseteq A.
    \end{align*}
    This concludes the proof.
\end{proof}

As a corollary of the foregoing, we can deduce the IP-van der Waerden's theorem in $G^2$. In other words, we prove \cref{Main_VdW-2}.
\begin{proof}[Proof of \cref{Main_VdW-2}]
   It follows by \cref{thm_7.1.5} and \cref{thm_7.1.4}.
\end{proof}
Similarly, we can prove that for any finite partition $G^2 = \bigcup_{j=1}^{r} C_j$, we have that there exists $1 \leq j \leq r$ such that
\begin{align*}
    S_2(\chi)&=\{ g \in G :  \exists x \in G^2 \text{ such that } \{ x, \mathcal{L}_g^{(1)} x, \mathcal{L}_g^{(1)}\mathcal{L}_g^{(2)}x\} \subseteq C_j \}\\
    &=\{ g \in G : \exists (a,b) \in G^2 \text{ such that } \{ (a,b),(ga , b),(ga , gb )\} \subseteq C_j \}
\end{align*}
is right IP$^{\ast}$.

\subsection{IP-van der Waerden Theorem for FC-groups}
\label{subsec:vdw group}
We give a sufficient condition for \cref{conjecture_2} to be true for $G^m$ with $m \in \mathbb{N}$. We define the concept of a family $\mathcal{W}$ of sets with the van der Waerden property and of groups with that property, namely van der Waerden groups. Thanks to those concepts, which will be later presented and defined, we prove \cref{thm_7.2.1}. More precisely, we prove that for any IP set $X$ and any $A \in \mathcal{W}$ there exists an IP set $Y \subseteq X$ such that for all $g \in Y$ there exists $B=B(g) \in \mathcal{W}$ such that for all $b \in B$ we have that $\{ b, \mathcal{R}_g^{(1)} b , \ldots, \mathcal{R}_g^{(1)} \ldots \mathcal{R}_g^{(k)} \} \subseteq A$ is equivalent to the fact that $D_k(\chi)$ is an IP$^{\ast}$ for every finite coloring $\chi$. Consequently, we derive a color-focusing lemma quite analogous to the one presented in \cref{subsec:vdw amenable dimension 2} with \cref{thm_7.1.3}. The main difference is that te sets in the van der Waerden family replace the sets with positive density in our arguments, and that the set $D_k(\chi)$ plays the role of $D=\{ g \in G : \overline{d}_{\Phi}(A \cap A(g^{-1},e_G) >0 \}$ when $m >2$. Therefore, by induction on $k$, we prove the IP-van der Waerden Theorem is true for van der Waerden groups.
Finally, we prove that FC-groups are van der Waerden groups by showing that the family of piecewise syndetic sets possesses the van der Waerden property. 
\begin{definition}\label{syndetic-thick-pws}
Let $G$ be a group.
\begin{enumerate}
    \item  A set $S \subseteq G$ is said to be \textbf{right syndetic} if there exists some finite $H \subseteq G$ such that 
 \begin{align*}
     G = \bigcup_{h \in H} h^{-1} S
 \end{align*}
 \item  A set $T \subseteq G$ is said to be \textbf{right thick} if for all finite $F \subseteq G$ exists $ x \in G $ such that $Fx \subseteq T$.
 \item  A set $A \subseteq G$ is said to be \textbf{right piecewise syndetic} if there exists some finite $H \subseteq G$ such that 
 \begin{align*}
     \bigcup_{h \in H}  h^{-1}A
 \end{align*}
 is right thick.
\end{enumerate}
\end{definition}
We have symmetric notions with respect to inversion of left syndetic sets, left thick sets and left piecewise syndetic sets. 
\begin{definition}\label{def_almost_vdw}
    Let $\mathcal{W}$ be a non-empty family of subsets of $G$ such that $\emptyset \not\in \mathcal{W}$. We say that $\mathcal{W}$ possesses the almost right van der Waerden property if for all $A \in \mathcal{W}$ either 1. or 2. is true.
    \begin{enumerate}
        \item There exists a right syndetic set $S$ such that for all finite $F \subseteq S$ we have that 
        \begin{align*}
           A_F= \bigcap_{f \in F} Af^{-1} \in \mathcal{W}.
        \end{align*}
        \item There exists a left thick set $T$ such that for all finite $F \subseteq T$ we have that
        \begin{align*}
            A_F = \bigcap_{f \in F} Af^{-1} \in \mathcal{W}.
        \end{align*}
    \end{enumerate}
\end{definition}

\begin{definition}\label{def_vdw_family}
    Let $\mathcal{W}$ be a non-empty family of subsets of $G$ such that $\emptyset \not\in \mathcal{W}$. 
    We say that $\mathcal{W}$ possesses the right van der Waerden property if the following properties are satisfied. 
    \begin{enumerate}
        \item $\mathcal{W}$ possesses the almost right van der Waerden property.
        \item $\mathcal{W}$ is partition regular.
        \item $\mathcal{W}$ is closed under the right translation, i.e. for all $A \in \mathcal{W}$ and for all $x \in G$ we have that $Ax \in \mathcal{W}$.
    \end{enumerate}
\end{definition}

The reason for requiring that a family $\mathcal{W}$ of sets having the right van der Waerden property has to be partition regular and closed under the right translation, i.e. right invariant, is that $\mathcal{W}$ sets plays the same role as positive upper density sets and therefore must possess at least the same essential properties.

\begin{definition}\label{def_vdw_group}
    Let $G$ be a countable, discrete, and amenable group. $G$ is said to be a right Van der Waerden group if for all $m \in \mathbb{N}$, we have that $G^m$ possesses a family of subsets $\mathcal{W}$ which has the right Van der Waerden property.
\end{definition}
Almost left van der Waerden property, left van der Waerden property, and left van der Waerden groups are defined similarly. 
We require that for each $m$, $G^m$ possess a family that has the van der Waerden property since this allows us to do the induction on $m$. Note that requiring that the group is amenable is not too restrictive since in \cite{free} V. Bergelson and N. Hindman have provided a counterexample for a free group either in $G^2$ or in $G^3$.
\begin{theorem}\label{thm_7.2.1}
Let $G$ be a right van der Waerden group. Let $\mathcal{W}$ be a family of subsets of $G^m$ which has the right Van der Waerden property. Let $1 \leq k < m$. Then the following are equivalent
\begin{enumerate}
    \item $D_k(\chi)$ is a left IP$^{\ast}$ set for any finite coloring $\chi : G^m \to \{1 ,\ldots , r\}$ of $G^m$.
    \item  Let $X$ be any left IP set, let $T \subseteq G^m$ be any left thick set and $S \subseteq G^m$ be any right syndetic set. Then there exist left IP sets $Y_T, Y_S \subset X$ such that for all $g \in Y_T$ there exist $y=y(g),x=x(g) \in G$ such that 
    \begin{align*}
        y\mathfrak{R}_{x,g,(1,\ldots,k)}=\mathfrak{R}_{yx,g,(1,\ldots,k)} \subseteq T.
    \end{align*}
    and for all $g \in Y_S$ there exists $h=h(g) \in G$ such that
    \begin{align*}
        h\mathfrak{R}_{x,g,(1,\ldots,k)}=\mathfrak{R}_{hx,g,(1,\ldots,k)} \subseteq S.
    \end{align*}
    \item Let $X$ be any left IP set, and let $A \in \mathcal{W}$ be arbitrary, then there exists a left IP set $Y \subset X$ such that for all $g \in Y$ there exists $B=B(g) , \tilde{B}=\tilde{B}(g) \in \mathcal{W}$ such that for all $b \in B$ we have that  \begin{align*}
        \{ b, \mathcal{R}_g^{(1)} b,\ldots,  \mathcal{R}_g^{(1)} \ldots \mathcal{R}_g^{(k)} b  \} \subseteq A 
    \end{align*}
    and for all $\tilde{b} \in \tilde{B}$ 
     \begin{align*}
         \{  \mathcal{R}_g^{(1)} \tilde{b}, \mathcal{R}_g^{(1)} \mathcal{R}_g^{(2)} \tilde{b} ,\ldots,  \mathcal{R}_g^{(1)} \ldots \mathcal{R}_g^{(k+1)} \tilde{b}  \} \subseteq A
   \end{align*}
     \end{enumerate}
\end{theorem}

\begin{proof}
 We proceed by proving that $1. \Rightarrow 2. \Rightarrow 3. \Rightarrow 1.$
\begin{enumerate}
    \item $\Rightarrow 2$.
Let $T \subseteq G^m$ be any left thick set. According to 1. we have that $D_{k}(\chi)$ is a left IP$^{\ast}$ set for any finite coloring $\chi$. Fix arbitrarily one coloring $\chi$. Hence let $X$ be any left IP set we have that $Y_T:=D_{k}(\chi) \cap X \subset X$ is non-empty. In particular, $Y_T$ is a left IP set, and for all $g \in Y_T$ we can find $x=x(g) \in G^m$ such that the finite set $\mathfrak{R}_{x,g,(1,\ldots,k)}$ is $\chi$-monochromatic. Since $T$ is left thick there exists some $y=y(g)$ such that \begin{align*}
        y\mathfrak{R}_{x,g,(1,\ldots,k)}=\mathfrak{R}_{yx,g,(1,\ldots,k)} \subseteq T.
    \end{align*}
    In particular there exists $z=z(g) \in G^m$, where $z=yx$, such that $\mathfrak{R}_{z,g,(1,\ldots,k)} \subseteq T$. \newline 
    Let $S \subseteq G^m$ be any right syndetic set. Then by definition, there exists a finite set $H \subset G^m$ such that 
\begin{align*}
    G^m = \bigcup_{h \in H} h^{-1} S.
\end{align*}
We can interpret this as a new finite coloring $\tilde{\chi}$ of $G^m$ using at most $r=\left| H \right|$ colors. According to 1. we have $Y_S:=D_{k}(\tilde{\chi}) \cap X$ is a left IP set. Hence for all $g \in Y_S$ we can find $ h=h(g) \in H$ and $x=x(g) \in G^m$ such that $\mathfrak{R}_{x,g,(1,\ldots,k)} \subseteq h^{-1}S$. Then \begin{align*}
        h\mathfrak{R}_{x,g,(1,\ldots,k)}=\mathfrak{R}_{hx,g,(1,\ldots,k)} \subseteq S.
    \end{align*}
     In particular there exists $z=z(g) \in G^m$, where $z=hx$, such that $\mathfrak{R}_{z,g,(1,\ldots,k)} \subseteq S$.
    \item $\Rightarrow 3$. Let $X$ be any left IP set and let $A \in \mathcal{W}$, since $\mathcal{W}$ has the almost right Van der Waerden property there exists either a right syndetic set $S$ such that for all finite $F \subseteq S$ we have that 
        \begin{align*}
           A_F= \bigcap_{f \in F} Af^{-1} \in \mathcal{W},
        \end{align*}
        either a left thick set $T$ such that for all finite $F \subseteq T$ we have that
        \begin{align*}
            A_F = \bigcap_{f \in F} Af^{-1} \in \mathcal{W}.
        \end{align*}
        Assume without loss of generality that there exists a right syndetic set with the above property. According to 2. there exists a left IP set $Y \subset X$ such that for all $g \in Y$ there exist $z=z(g) \in G^m$ such that 
    \begin{align*}
        \mathfrak{R}_{z,g,(1,\ldots,k)} \subseteq S.
    \end{align*}
Take as a finite set $F=\mathfrak{R}_{z,g,(1,\ldots,k)} \subset S$. For the sake of simplicity, let's denote $z_0=z$ and $z_j = \mathcal{R}_g^{(1)} \ldots \mathcal{R}_g^{(j)} z \in \mathfrak{R}_{z,g,(1,\ldots,k)}$. Hence we have that
    \begin{align*}
        A_F = \bigcap_{j=0}^{k} A z_j^{-1} \in \mathcal{W}.
    \end{align*}
     Let denote $(e_G,\ldots, e_G)=g_0 $ the neutral element of $G^m$, and $g_j = \mathcal{R}_{g}^{(1)} \ldots \mathcal{R}_{g}^{(j)} g_0 \in G^m $. Note that for all $0 \leq j \leq k$ we have that $z_j=z_0 g_j$, hence we have that
    \begin{align*}
        B = \bigcap_{j=0}^{k} A g_j^{-1} \in \mathcal{W}
    \end{align*}
    since we have $B=A_Fz_0$ and $\mathcal{W}$ is closed under right translation.
    Now it is straightforward to check that for all $b \in B$ \begin{align*}
         \{b,  \mathcal{R}_g^{(1)} b, \mathcal{R}_g^{(1)} \mathcal{R}_g^{(2)} b ,\ldots,  \mathcal{R}_g^{(1)} \ldots \mathcal{R}_g^{(k)} b  \} \subseteq A.
   \end{align*}
    In accordance with \cref{lemma_20.1} and using point 2. we can deduce that there exists $B \in \mathcal{W}$ such that for all $b\in B$ we have $\mathfrak{R}_{b,g,(2,\ldots,k+1)} \subseteq A$. Let $\tilde{B} = Bg_1^{-1}$, where $g_1 = \mathcal{R}_g^{(1)} g_0$, again by closure under right translation we have that $\tilde{B} \in \mathcal{W}$. Using the fact that $\mathcal{R}^{(1)} \mathcal{R}^{(j)} = \mathcal{R}^{(j)} \mathcal{R}^{(1)}$ for all $1 <j$, then for all $\tilde{b} \in \tilde{B}$, we have
    \begin{align*}
         \{  \mathcal{R}_g^{(1)} \tilde{b}, \mathcal{R}_g^{(1)} \mathcal{R}_g^{(2)} \tilde{b} ,\ldots,  \mathcal{R}_g^{(1)} \ldots \mathcal{R}_g^{(k+1)} \tilde{b}  \} \subseteq A.
   \end{align*}
   \item $\Rightarrow 1.$ Let $X$ be an arbitrary left IP set, if $G^m$ is colored using finitely many colors in accordance to some finite coloring $\chi : G^m \to \{1,\ldots, r\}$, then since $\mathcal{W}$ is partition regular, we have that there exists a $\chi$-monochromatic set $A \in \mathcal{W}$. Hence according to 3. we have that there exists a left IP set $ Y \subset X$, such that for all $g \in Y$, there is $B=B(g) \in \mathcal{W}$ such that for any $b \in B \subset G^m$, we have  
   \begin{align*}
        \mathfrak{R}_{b,g,(1,\ldots,k)} = \{b,  \mathcal{R}_g^{(1)} b, \mathcal{R}_g^{(1)} \mathcal{R}_g^{(2)} b ,\ldots,  \mathcal{R}_g^{(1)} \ldots \mathcal{R}_g^{(k)} b  \} \subseteq A.
   \end{align*}
   which is monochromatic with respect to $\chi$, and this imply that $D_{k}(\chi) \cap X \neq \emptyset$. Since $X$ was an arbitrary left IP set, we have that $D_{k}(\chi)$ intersects non-trivially every left IP set and thus it is a left IP$^{\ast}$ set.
   \end{enumerate}
    
\end{proof}

The following is the analog of \cref{thm_7.1.3} given in \cref{subsec:vdw amenable dimension 2}. The proof is essentially the same, since the only important property used by sets of positive density is that they are partition regular.
\begin{theorem}[Color focusing]\label{color_focusing}
    Assume that for any finite coloring $\chi$ of $G^m$ and for $1 \leq k-1 < m$ we have that $D_{k-1}(\chi)$ is left IP$^{\ast}$, then for any $r \in \mathbb{N}$ we have that for any left IP set $X$, there exists $\chi$-monochromatic sets $ A_0,\ldots, A_r \in \mathcal{W}$ and left IP sets $ Y_{r-1} \subset \ldots \subset Y_0 \subset X$ such that for all $0 \leq i < j \leq r$ there exists $u_{j,i} \in Y_i$ with
    \begin{align*}
        \{ \mathcal{R}_{u_{j,i}}^{(1)} a , \ldots , \mathcal{R}_{u_{j,i}}^{(1)} \ldots \mathcal{R}_{u_{j,i}}^{(k)} a : a \in A_j \} \subseteq A_i.
    \end{align*}
\end{theorem}
\begin{proof}
This is essentially a repetition of \cref{thm_7.1.3} where sets in a van der Waerden family of $G^m$ are replacing sets of positive upper density.
\end{proof}

The following is the analog of \cref{thm_7.1.4} \cref{subsec:vdw amenable dimension 2}. But in contrast to \cref{thm_7.1.4} we do induction on $k<m$, since we know that $D_2(\chi)$ is IP$^{\ast}$.
\begin{theorem}\label{thm_7.2.2} 
Let $G$ be a discrete, countable group. Let $\chi : G^m \to \{ 1, \ldots, r\}$ be any finite coloring of $G^m$. And let $1 \leq k-1 < m $.  Suppose that $D_{k-1}(\chi)$ is a left IP$^{\ast}$ set for all $\chi$, then $D_k(\chi)$ is a left IP$^{\ast}$ set for all $\chi$.
 \end{theorem}
 \begin{proof}
     This is essentially a repetition of \cref{thm_7.1.4}, where \cref{color_focusing} is used instead of \cref{thm_7.1.3}
 \end{proof}
 \begin{theorem}\label{upgrade_thm}
    Let $G$ be a countable discrete and amenable group and $m \in \mathbb{N}$ be arbitrary. Let $\chi : G^m \to \{1 ,\ldots , r\}$ be any given finite coloring.
    \begin{enumerate}
        \item If $G$ is a right Van der Waerden group then $D_m(\chi)$ is left IP$^{\ast}$.
        \item If $G$ is a left Van der Waerden group then $S_m(\chi)$ is right IP$^{\ast}$.
    \end{enumerate}
\end{theorem}
\begin{proof}
    This follows from \cref{Main_VdW-2} and \cref{thm_7.2.2}.  
\end{proof}
Let's remark that a group is right van der Waerden group if and only if it is a left van der Waerden group. Indeed $\mathcal{W}_1$ has the right van der Waerden property if and only if $\mathcal{W}_1^{-1}$ has the left van der Waerden property.
Let us recall that for a left piecewise syndetic set $A$ there exists a return set $L_A=\{x \in G : Ax^{-1} \in p\}$ which is a left syndetic set (see \cite[Theorem 4.39 on page 101]{idem}). For the family of left piecewise syndetic sets to possess the right van der Waerden property, it is sufficient that $L_A$ is right syndetic or left thick. Since FC-groups possess the property that every left syndetic sets are also right syndetic sets (see \cite{436093}), we deduce that FC-groups are van der Waerden groups.

 \begin{theorem}\label{FC-are VDW groups}
     Let $G$ be an FC-group. Then $\mathcal{W}=\{ A \subseteq G^m : A \text{ is left piecewise syndetic set}  \} $ has the right van der Waerden property in $G^m$. That is, $G$ is both a right and left van der Waerden group.
 \end{theorem}
 \begin{proof}
     Left piecewise syndetic sets are partition regular and are closed under right translation. Moreover, let $A \in \mathcal{W}$, take $p\in K(\beta (G^m), \bullet) \cap \overline{A}$, where $K(\beta (G^m), \bullet)$ denote the minimal two-sided ideal of $(\beta (G^m),\bullet)$. Since $G$ is an FC-group we have that also $G^m$ is an FC-group. We have that $L_A=\{x \in G^m : Ax^{-1} \in p \}$ is right syndetic in $G^m$ and thus since ultrafilters are closed under finite intersection, for every finite $F \subseteq L_A$ we have that
     \begin{align*}
         A_F = \bigcap_{f \in F} A f^{-1} \in p.
     \end{align*}
     Then we have that $p \in \overline{A_F} \cap K(\beta (G^m),\bullet)$ and we conclude that $A_F$ is left piecewise syndetic. Hence $A_F \in \mathcal{W}$. 
     
 \end{proof}

 \begin{proof}[Proof of \cref{van-der-waerden-holds VDW group}]
     It follows from \cref{upgrade_thm} and \cref{FC-are VDW groups}.
 \end{proof} 
 We proved that in the class of FC-groups the family $\mathcal{W}_1$ of left piecewise syndetic sets possesses the right van der Waerden property. To reach this conclusion, we took advantage of the fact that in FC-groups left syndetic sets and right syndetic sets coincide. As a consequence of this fact and a result of V. Paulsen in \cite{Fc-amenable}, we have that FC-groups are amenable.
However, to extend this result to general amenable groups, we faced a significant obstacle: left and right syndetic sets are not necessarily equals, this even fails in nilpotent groups. To better understand this challenge, let us consider the two ways to extend the group operation to the Stone-Čech compactification semigroup, $(\beta G, \bullet)$ and $(\beta G, \circ)$, making it a left or right topological compact semigroup respectively. In the first case, for all $p,q \in \beta G$ we have that $A \in p \bullet q $ if and only if $\{x \in G : Ax^{-1} \in p\} \in q $, while in the second case for all $p,q \in \beta G$ we have that $A \in p \circ q $ if and only if $\{x \in G : x^{-1}A \in q\} \in p$. If we consider $\mathcal{W}_1$ to be the family of left piecewise syndetic sets, it is relatively simple to show that in the left topological semigroup, for a left syndetic piecewise set $A \in \mathcal{W}_1$, we have an ultrafilter $p$ in the minimal two-sided ideal of $(\beta G, \bullet)$ containing $A$ and such that $L_A=\{ x \in G: Ax^{-1} \in p \}$ is left syndetic. In FC-groups, taking advantage of the fact that the concepts of left syndetic sets and right syndetic sets coincide, we can conclude that, due to the closure of ultrafilters with respect to finite intersection, the family of left piecewise syndetic sets possesses the right van der Waerden property. However, without this fundamental property on syndetic sets, there is no guarantee that $L_A$ is also right syndetic. If, on the other hand, we consider the family $\mathcal{W}_2$ of right piecewise syndetic sets and $A \in \mathcal{W}_2$, it is true that there exists an ultrafilter $p$ in the minimal two-sided ideal of $(\beta G, \circ)$ containing $A$ and such that $R_A=\{ x \in G: x^{-1}A \in p \}$ is right syndetic. However, the shift occurs on the other side and therefore nothing guarantees us the existence of a right syndetic set $S$ such that for all finite $F \subseteq S$ we have that
    \begin{align*}
        A_F = \bigcap_{f \in F} Af^{-1} \in \mathcal{W}.
    \end{align*}
The difficulty in addressing this question arises from the peculiar characteristics of operations in the topological left and right semigroups of the Stone-Čech compactification, where objects behave in a way that follows the opposite "side" from the one we would need. Therefore, we have shown that if it is true that in an amenable group $G$ we have that for every left piecewise syndetic set $A$ there exists an ultrafilter $p$ in the minimal two-sided ideal of $(\beta G, \bullet)$ such that $L_A = \{ x \in G : Ax^{-1} \in p\}$ is right syndetic, then this implies IP-van der Waerden Theorem holds in $G^m$ with $m$ arbitrary. We looked for a counterexample to this stronger proposition, without success. We have found examples of sets $X$ that are left syndetic, not left thick, and not right syndetic in the amenable group $\mathbb{Z}_2 \ast \mathbb{Z}_2 = \left< a,b \mid a^2=b^2=1\right>$ but we also showed that those examples cannot be a set of the form $L_A$ for some left piecewise syndetic set $A$. \newline
We don't know the answer to the following:
\begin{question}
Let $G$ be an amenable group
    \begin{enumerate}
        \item Does the family $\mathcal{W}_1=\{ A \subseteq G^m : A \text{ is left piecewise syndetic set}  \} $ have the right van der Waerden property in $G^m$?
        \item Let $A$ be a left piecewise syndetic set. Let $p \in (\beta G, \bullet)$ be such that $A \in p$. Is $\{ x \in G : Ax^{-1} \in p\} $ right syndetic?
        \item Does the family $\mathcal{W}_3=\{ A \subseteq G^m : A \text{ is both left and right piecewise syndetic set}  \} $ have both the right and the left van der Waerden property in $G^m$?
        \item Is $G=\mathbb{Z}_2 \ast \mathbb{Z}_2$ a van der Waerden group?
        \item Does there exists an amenable group which is not a van der Waerden group?

    \end{enumerate}
\end{question}

\subsection{Acknowledgments}

This paper is the shorter version of my master's thesis at École polytechnique fédérale de Lausanne. I want to thank my esteemed supervisors, Ethan Ackelsberg and Prof. Florian K. Richter, for their invaluable teachings and guidance throughout this study. I am also deeply appreciative of Prof. Vitaly Bergelson for kindly sharing his expertise in my work and for his valuable feedback.



\begin{thebibliography}{99}

    \bibitem{Austin2013NonconventionalEA}
    T. Austin,
    \textit{Non-conventional ergodic averages for several commuting actions of an amenable group},
    Journal d'Analyse Mathématique, 130 (2013), 243--274.
    \href{https://api.semanticscholar.org/CorpusID:119325118}{URL}

    \bibitem{free}
    V. Bergelson, and N. Hindman,
    \textit{Some topological semicommutative van der Waerden type theorems and their combinatorial consequences},
    J. London Math. Soc. (2), 45 (1992), no. 3, 385--403.
    \href{https://doi.org/10.1112/jlms/s2-45.3.385}{DOI}

    \bibitem{shift-pws}
    V. Bergelson, N. Hindman, and R. McCutcheon,
    \textit{Notions of size and combinatorial properties of quotient sets in semigroups},
    Proceedings of the 1998 Topology and Dynamics Conference (Fairfax, VA), Topology Proc., 23 (1998), 23--60.

        \bibitem{discordant}
    V. Bergelson, J. Huryn, R. Raghavan,
    \textit{Discordant sets and ergodic Ramsey theory},
    Involve, 15 (2022), no. 1, 89--130.
    \href{https://doi.org/10.2140/involve.2022.15.89}{DOI}

    \bibitem{Bergelson2015NewPA}
    V. Bergelson, J. H. Johnson, and J. Moreira,
    \textit{New polynomial and multidimensional extensions of classical partition results},
    J. Comb. Theory A, 147 (2015), 119--154.
    \href{https://api.semanticscholar.org/CorpusID:42673273}{URL}

    \bibitem{Central}
    V. Bergelson, and R. McCutcheon,
    \textit{Central sets and a non-commutative Roth theorem},
    Amer. J. Math., 129 (2007), no. 5, 1251--1275.
    \href{https://doi.org/10.1353/ajm.2007.0031}{DOI}

    \bibitem{density-folner}
    V. Bergelson, R. McCutcheon, and Q. Zhang,
    \textit{A Roth theorem for amenable groups},
    Amer. J. Math., 119 (1997), no. 6, 1173--1211. \href{http://muse.jhu.edu/journals/american_journal_of_mathematics/v119/119.6bergelson.pdf}{URL}


    \bibitem{Bergelson1996ErgodicRT}
    V. Bergelson,
    \textit{Ergodic Ramsey Theory–an Update},
    1996.
    \href{https://api.semanticscholar.org/CorpusID:5698999}{URL}

      \bibitem{ultrafilter-partition-regularss}
    W. Brian, and P. Oprocha,
    \textit{Ultrafilters and Ramsey-type shadowing phenomena in topological dynamics},
    Israel J. Math., 227 (2018), no. 1, 423--453.
    \href{https://doi.org/10.1007/s11856-018-1739-4}{DOI}
 
    \bibitem{436093}
    Y. de Cornulier,
    \textit{Left syndeticity and right syndeticity in nilpotent group},
    MathOverflow, 2022.
    \href{https://mathoverflow.net/q/436093}{URL}
    
    \bibitem{Furstenberg}
    H. Furstenberg,
    \textit{Recurrence in Ergodic Theory and combinatorial Number Theory},
    Princeton University Press, 1981.

    \bibitem{Furstenberg1978TopologicalDA}
    H. Furstenberg, and B. Weiss,
    \textit{Topological dynamics and combinatorial number theory},
    Journal d’Analyse Mathématique, 34 (1978), 61--85.
    \href{https://api.semanticscholar.org/CorpusID:121646951}{URL}

    \bibitem{idem}
    N. Hindman, and D. Strauss,
    \textit{Algebra in the Stone-Čech compactification. Theory and Applications},
    de Gruyter Exp. Math. vol 27, de Gruyter, Berlin, 1998.
  
    \bibitem{FC-group}
    H. Meyn, 
    \textit{FC-groups and related classes},
    Rend. Sem. Mat. Univ. Padova, 47 (1972), 65--75.
    \href{http://www.numdam.org/item?id=RSMUP_1972__47__65_0}{URL}

    \bibitem{Fc-amenable}
    V. I. Paulsen,
    \textit{Syndetic sets and amenability},
    Proc. Amer. Math. Soc., 140 (2012), no. 6, 1997--2001.
    \href{https://doi.org/10.1090/S0002-9939-2011-11247-4}{DOI}   

    \bibitem{Rado1945NoteOC}
    R. Rado,
    \textit{Note on Combinatorial Analysis},
    Proceedings of The London Mathematical Society, 1945, 122--160.
    \href{https://api.semanticscholar.org/CorpusID:121224215}{URL}

    \bibitem{original-work}
    B. L. van der Waerden,
    \textit{Wie der Beweis der Vermutung von Baudet gefunden wurde},
    Elemente der Mathematik, 53 (1998), no. 4, 139--148.
    \href{https://doi.org/10.1007/s000170050045}{DOI}


    \bibitem{Witt}
    E. Witt,
    \textit{Ein kombinatorisches Satz der Elementargeometrie},
    Mathematische Nachrichten, 1951, 261--262.

\end{thebibliography}
\end{document}